\documentclass[a4paper,11pt, reqno]{amsart}

\usepackage{amsmath,amssymb,graphicx}
\usepackage{hyperref}
\usepackage{cancel}
\usepackage{soul}
\usepackage[utf8]{inputenc}
\usepackage[english]{babel}
\usepackage[normalem]{ulem}
 
\usepackage[dvipsnames]{xcolor}
\usepackage{cancel}
\usepackage{dynkin-diagrams}
\usepackage{tikz}
\usetikzlibrary{arrows.meta}
\usepackage{graphicx}
\usepackage[left=1.12in, right=1.12in, top=1.5in, bottom=1.5in]{geometry}
\usepackage{caption}
\usepackage{amsfonts}
\usepackage{arydshln}
\usepackage{appendix}
\usepackage{amsmath,amssymb,verbatim,mathrsfs,latexsym,paralist}
\usepackage{graphicx}
\usepackage{epstopdf}
\usepackage{indentfirst}
\usepackage{multirow}
\usepackage{stmaryrd}
\usepackage{amsthm}
\usepackage{longtable}
\usepackage{float}
\usepackage{cancel}
\allowdisplaybreaks[4]

\numberwithin{equation}{section}

\newtheorem{theorem}{Theorem}[section]
\newtheorem{lemma}[theorem]{Lemma}
\newtheorem{proposition}[theorem]{Proposition}
\newtheorem{example}[theorem]{Example}

\newtheorem{definition}{Definition}[section]

\theoremstyle{plain}

\newtheorem*{remark}{Remark}

\DeclareMathOperator{\End}{End}

\DeclareMathOperator{\Id}{Id}

\DeclareMathOperator{\re}{re}
\DeclareMathOperator{\im}{im}

\title[Wide Subalgebras of Kac--Moody Algebras]
{Wide Regular Subalgebras of Symmetrizable Kac--Moody Algebras and an Extension of Schur's Lemma}

\author[A. Douglas]{Andrew Douglas\textsuperscript{1,2,3}}

\address{\textsuperscript{1}Department of Mathematics, New York City College of Technology,
City University of New York, Brooklyn, NY 11201, USA}

\address{\textsuperscript{2}Ph.D.\ Programs in Mathematics and Physics, Graduate Center,
City University of New York, New York, NY 10016, USA}

\address{\textsuperscript{3}Department of Mathematics, University of Toronto,
Toronto, ON M5S 2E4, Canada}

\email{adouglas@citytech.cuny.edu}

\author[A. Ali]{Abid Ali\textsuperscript{4}}

\address{\textsuperscript{4}Department of Mathematics and Statistics, University of Saskatchewan,
Saskatoon, SK S7N 5E6, Canada}

\begin{document}

\keywords{Symmetrizable Kac--Moody algebras,
Regular subalgebras,
Wide subalgebras,
Integrable highest weight modules,
Root closure,
Schur's lemma}

\subjclass[2020]{17B05, 17B10, 17B20, 17B22, 17B67}

\begin{abstract}
The behavior of representations under restriction is a central theme in Lie theory. We study wide regular subalgebras of symmetrizable Kac--Moody algebras, extending work of Douglas and Repka on semisimple Lie algebras. A subalgebra is \emph{wide} if every irreducible integrable highest weight module remains indecomposable upon restriction. Let $\mathfrak{g}$ be a symmetrizable Kac--Moody algebra with Cartan subalgebra $\mathfrak{h}$, root system $\Phi$, simple roots $\Pi$, and root space decomposition $\mathfrak{g}=\mathfrak{h}\oplus\bigoplus_{\alpha\in\Phi}\mathfrak{g}_\alpha$. Denote by $\Phi_{\re}$ the set of real roots. To a regular subalgebra $\mathfrak{s}$ normalized by $\mathfrak{h}$, we associate a subset $T\subseteq \Phi$ by declaring $\alpha\in T$ if $\mathfrak{s}\cap \mathfrak{g}_\alpha\ne \{0\}$. One of our main results is an extension of Schur's lemma: if $\mathfrak{h}\subseteq \mathfrak{s}$ and the real closure of $(T\cup(-T))\cap \Phi_{\re}$ contains $\Pi$, then $(\End V)^{\mathfrak{s}}=\mathbb{C}\Id_V$ for every irreducible integrable highest weight module $V$. As a consequence, this real-root closure condition yields a sufficient condition for wideness. In the affine case, we establish a necessary condition for wideness: if $\mathfrak{s}$ is wide, then the closure of $T\cup(-T)$ in $\Phi$ is all of $\Phi$, and this implication holds without assuming that $\mathfrak{h}\subseteq \mathfrak{s}$. A key ingredient is a structural result showing that closed subsets of affine root systems are closed under arbitrary finite root sums that remain roots.
\end{abstract}

\dedicatory{Dedicated to the memory of Joe Repka.}

\maketitle

\section{Introduction}\label{intro}

The behavior of representations under restriction is a central theme in Lie theory. A fundamental question is to determine when an irreducible module remains indecomposable upon restriction to a subalgebra. We examine wide regular subalgebras of symmetrizable Kac--Moody algebras, extending 
work of Douglas and Repka on semisimple Lie algebras~\cite{DouglasRepka2025}. Let $\mathfrak{g}$ be a symmetrizable Kac--Moody algebra with Cartan subalgebra $\mathfrak{h}$, root system $\Phi$, simple roots $\Pi$, and root space decomposition
\[
\mathfrak{g} = \mathfrak{h} \oplus \bigoplus_{\alpha \in \Phi} \mathfrak{g}_\alpha.
\]
Denote by $\Phi_{\re}$ the set of real roots. A subalgebra is \emph{wide} if every irreducible integrable highest weight $\mathfrak{g}$-module remains indecomposable upon restriction.

The theory of regular subalgebras goes back to the foundational work of Dynkin~\cite{Dynkin1952}. The study of wide subalgebras dates at least to the work of Douglas and Premat~\cite{DouglasPremat2007}, and is closely connected to classical topics such as branching rules. The term \emph{wide subalgebra} was introduced by Panyushev~\cite{panyu} in 2014. Most relevant to the present work, Douglas and Repka~\cite{DouglasRepka2025} established necessary and sufficient conditions for regular subalgebras of semisimple Lie algebras to be wide.

In this article, we extend aspects of this theory to the symmetrizable Kac--Moody setting. To a regular subalgebra $\mathfrak{s}$ normalized by $\mathfrak{h}$, we associate a subset $T\subseteq \Phi$ by declaring that $\alpha\in T$ if $\mathfrak{s}\cap \mathfrak{g}_\alpha\ne \{0\}$.

One of our main results is an extension of Schur's lemma: if $\mathfrak{h}\subseteq \mathfrak{s}$ and the real closure of $(T\cup(-T))\cap \Phi_{\re}$ contains $\Pi$, then $(\End V)^\mathfrak{s}= \mathbb{C} \Id_V$, where 
\[(\End V)^{\mathfrak{s}} \coloneqq \{\, F \in \End(V) \mid x \cdot F=0 \text{ for all } x \in \mathfrak{s} \},\]
and $V$ is an irreducible integrable highest weight $\mathfrak{g}$-module. As a consequence, this real-root closure condition yields a sufficient condition for wideness.

In the affine case, we establish a necessary condition for wideness: if $\mathfrak{s}$ is wide, then the closure of $T\cup(-T)$ in $\Phi$ is all of $\Phi$, and notably this implication does not require the assumption that $\mathfrak{h}\subseteq \mathfrak{s}$. This relies on a structural result showing that closed subsets of affine root systems are closed under arbitrary finite root sums that remain roots. In the finite-dimensional semisimple case, these closure conditions coincide, yielding a necessary and sufficient criterion for wideness, which does not require $\mathfrak{h}\subseteq \mathfrak{s}$ (cf.~\cite[Corollary 3.10]{DouglasRepka2025}).

Extending the results of Douglas and Repka to the symmetrizable Kac--Moody setting requires a different conceptual approach. In contrast to the finite-dimensional case, the root system contains both real and imaginary roots, there are infinitely many root spaces, and integrable highest weight modules typically have infinitely many weights. Moreover, $\End(V)$ is no longer completely reducible, and the classical arguments fail. We overcome these challenges by developing a new approach, centered on an extension of Schur's lemma that replaces complete reducibility and yields a sufficient condition for wideness. In the affine case, we establish a structural result on closed subsets of root systems that supports the proof of a necessary condition for wideness without requiring $\mathfrak{h}\subseteq \mathfrak{s}$.

The article is organized as follows. Section~\ref{background} reviews background material on symmetrizable Kac--Moody algebras, their regular subalgebras, and their representations. In Section~\ref{lwidesec}, we study wide regular subalgebras, prove a sufficient condition for wideness for Cartan-regular subalgebras $\mathfrak{h}\subseteq \mathfrak{s}$ via the extension of Schur's lemma, and establish, in the affine case, a necessary condition for wideness. In this section, we also establish finite-sum closure in affine type, which is used to prove the necessary condition for wideness.  In Section~\ref{conclusion}, we summarize our results and look forward to related open directions of research.
 
\section{Background on Kac--Moody algebras}\label{background}

In this section, we review the necessary background on symmetrizable
Kac--Moody algebras, their representations, root systems, and the
theory of closed subsets of root systems and regular subalgebras
(cf.~\cite{Bourbaki, CarboneCoelhoMurrayThurmanZhu2025, Carter, FeliksonRetakhTumarkin2008, HongKang, Kac90, KacRaina, Ku, MP,Naito1992}).

\subsection{Kac--Moody algebras}

Let $A=(a_{ij})$ be an $\ell\times \ell$ \emph{generalized Cartan matrix} (GCM), so that
\[
a_{ii}=2,\quad a_{ij}\in \mathbb{Z},\; a_{ij}\le 0 \ (i\neq j),\quad \text{and} \quad a_{ij}=0 \Rightarrow a_{ji}=0.
\]
The matrix $A$ is \emph{symmetrizable} if there exists a diagonal matrix
$D=\mathrm{diag}(d_i)$ with $d_i>0$ such that $DA$ is symmetric. 
In this article, all GCMs are  symmetrizable.

A \emph{realization} of $A$ is a triple $(\mathfrak{h},\Pi,\Pi^\vee)$, where
$\mathfrak{h}$ is a complex vector space and
\[
\Pi=\{\alpha_1,\dots,\alpha_\ell\}\subset \mathfrak{h}^*, 
\qquad
\Pi^\vee=\{\alpha_1^\vee,\dots,\alpha_\ell^\vee\}\subset \mathfrak{h}
\]
are linearly independent sets satisfying
\[
\alpha_j(\alpha_i^\vee)=a_{ij},
\qquad
\dim \mathfrak{h}= 2\ell-\operatorname{rank}(A).
\]
The elements of $\Pi$ and $\Pi^\vee$ are called the \emph{simple roots} and
\emph{simple coroots}, respectively.

The associated \emph{symmetrizable Kac--Moody algebra}
$\mathfrak{g}=\mathfrak{g}(A)$ is the Lie algebra over $\mathbb{C}$ generated by
$\mathfrak{h}$ together with elements $e_i,f_i$ ($1\le i\le \ell$),
called \emph{simple root vectors}, subject to the relations
\begin{align*}
[h,h']&=0,\\
[e_i,f_j]&=\delta_{ij}\,\alpha_i^\vee,\\
[h,e_i]&=\alpha_i(h)e_i,\qquad
[h,f_i]=-\,\alpha_i(h)f_i,\\
(\operatorname{ad} e_i)^{1-a_{ij}} e_j&=0,\qquad
(\operatorname{ad} f_i)^{1-a_{ij}} f_j=0 \quad (i\neq j),
\end{align*}
for all $h,h'\in \mathfrak{h}$. These are the \emph{Cartan}, \emph{Weyl}, and
\emph{Serre relations}.  The abelian subalgebra $\mathfrak{h}$ is a Cartan subalgebra of $\mathfrak{g}$.

An indecomposable symmetrizable generalized Cartan matrix $A$ is said to be of \emph{finite type}, \emph{affine type}, or \emph{indefinite type} according to whether its symmetric symmetrization $DA$ is positive definite, positive semidefinite of corank $1$, or neither. 
Equivalently, a Kac--Moody algebra $\mathfrak{g}$ is said to be of finite, affine, or indefinite type if its generalized Cartan matrix has the corresponding property.

Let $\Phi\subset \mathfrak{h}^*$ denote the corresponding root system of $\mathfrak{g}$.
Then $\mathfrak{g}$ admits the \emph{root space decomposition}
\[
\mathfrak{g}
=\mathfrak{h}\oplus \bigoplus_{\alpha\in\Phi}\mathfrak{g}_\alpha,
\qquad
\mathfrak{g}_\alpha
=\{x\in\mathfrak{g}\mid [h,x]=\alpha(h)x \ \text{for all } h\in\mathfrak{h}\}.
\]
The root system decomposes as
\[
\Phi=\Phi^+\cup \Phi^-,\qquad \Phi^-=-\Phi^+,
\]
where $\Phi^+$ consists of all roots that are nonnegative integer combinations of simple roots.
The generators $e_i$ and $f_i$ are root vectors in $\mathfrak{g}_{\alpha_i}$ 
and $\mathfrak{g}_{-\alpha_i}$, respectively.

For a symmetrizable Kac--Moody algebra $\mathfrak g$ with root system $\Phi$,
there exists an invariant symmetric bilinear form $(\cdot\mid\cdot)$ on $\mathfrak g$
whose restriction to $\mathfrak h$ is nondegenerate. Hence it induces a bilinear
form on $\mathfrak h^*$,
\begin{align*}
\mathfrak{h}^* \times \mathfrak{h}^* \rightarrow \mathbb{C},\quad
(\alpha,\beta) \mapsto (\alpha \mid \beta).
\end{align*}
Let $\nu:\mathfrak h\to \mathfrak h^*$ be the isomorphism induced by the
restriction of $(\cdot\mid\cdot)$ to $\mathfrak h$. Then
\[
\nu(\alpha_i^\vee)=\frac{2\alpha_i}{(\alpha_i\mid \alpha_i)},
\qquad\text{so that}\qquad
\alpha_j(\alpha_i^\vee)=\frac{2(\alpha_i\mid \alpha_j)}{(\alpha_i\mid \alpha_i)}=a_{ij}.
\]

\subsection{Representation theory of Kac--Moody algebras}
 A module $V$ of a Kac--Moody algebra is a \emph{weight module} if there is a weight space decomposition 
\[
V=\bigoplus_{\mu\in \mathfrak{h}^*} V_\mu,~ \text{where}~ V_\mu=\{v\in V  \mid h v=\mu(h)v, ~\text{for all}~ h\in \mathfrak{h}\}.
\]
If $V_\mu\ne \{0\}$, then $\mu$ is a \emph{weight} of $V$. The set of weights of $V$ is denoted $\mathrm{wts}(V)$.
A weight module $V$ is a \emph{highest weight module} of highest weight $\lambda\in \mathfrak{h}^*$ if there is a nonzero vector $v_\lambda\in V$, called the \emph{highest weight vector}, for which
\begin{align*}
    e_i v_\lambda &= 0 \quad \text{for all} \quad 1\le i\le \ell,\\
    h v_\lambda &= \lambda(h) v_\lambda \quad \text{for all}\quad h\in \mathfrak{h}, \quad \text{and}\\
    V&= U(\mathfrak{g})v_\lambda,
\end{align*}
where $U(\mathfrak{g})$ is the  universal enveloping algebra of $\mathfrak{g}$.
Such a module is \emph{irreducible} if it has no proper submodules. A module  is called \emph{indecomposable} if it cannot be
written as a direct sum of two nonzero submodules.

An element $x\in \mathfrak{g}$ is \emph{locally nilpotent} on a weight module $V$ if for any $v\in V$ there is a positive integer $N$ such that $x^N\cdot v=0$.
A weight module $V$ is \emph{integrable} if all $e_i$ and $f_i$, $1\le i \le \ell$, are locally nilpotent on $V$. 

This paper focuses on the \emph{irreducible integrable highest weight 
modules} of symmetrizable Kac--Moody algebras. Given an irreducible 
integrable highest weight module $V$, we may endow the endomorphism 
space $\End(V)$ with a $\mathfrak{g}$-module structure via
\[
x \cdot F = x \circ F - F \circ x, \qquad x \in \mathfrak{g},
\]
where the right-hand side denotes the commutator in $\End(V)$.

Note that $\End(V)$ is not necessarily completely reducible, 
which is fundamentally different from the semisimple setting, 
where $\End(V)$ \emph{is} completely reducible for every 
finite-dimensional irreducible module. 
This distinction is a considerable obstacle to extending the 
results of~\cite{DouglasRepka2025} from the semisimple setting 
to the symmetrizable Kac--Moody setting.

\subsection{Root systems: Real and imaginary roots}

The real roots $\Phi_{\re}$ and imaginary roots $\Phi_{\im}$ are given by
\[
\Phi_{\re}=\{\alpha\in \Phi \mid (\alpha \mid \alpha)>0 \},\qquad
\Phi_{\im}=\{\alpha \in \Phi \mid (\alpha \mid \alpha)\le 0 \}.
\]
An imaginary root $\alpha$ such that $(\alpha \mid \alpha)=0$ is called \emph{isotropic}.
If $\alpha \in \Phi_{\re}$, then $\dim(\mathfrak{g}_{\alpha})=1$.
If $\alpha \in \Phi_{\im}$, then the root space $\mathfrak{g}_{\alpha}$ is finite-dimensional (cf.~\cite{Kac90}).

In finite type, all roots are real, whereas in affine and indefinite types there exist infinitely many imaginary roots. In affine type, there exists a distinguished positive imaginary root $\delta$, called the \emph{null root}, characterized by
\[
(\delta\mid \delta)=0,
\]
and the property that every imaginary root is an integer multiple of $\delta$, that is,
\[
\Phi_{\im}=\{k\delta \mid k\in \mathbb{Z}\setminus\{0\}\}.
\]

Fix $\beta \in \Phi^+$, and define
\[
\mathfrak{n}_{\pm}^{(\beta)}
=
\bigoplus_{j\geq 1} \mathfrak{g}_{\pm j\beta},
\qquad
\mathfrak{g}^{(\beta)}
=
\mathfrak{n}_{-}^{(\beta)}
\oplus
\mathbb{C}\nu^{-1}(\beta)
\oplus
\mathfrak{n}_{+}^{(\beta)}.
\]
Then $\mathfrak{g}^{(\beta)}$ is a subalgebra of $\mathfrak{g}$.
If $\beta \in \Phi_{\re}$, then $\mathfrak{g}^{(\beta)}\cong \mathfrak{sl}_2$
and is the $\mathfrak{sl}_2$-subalgebra corresponding to the real root $\beta$.
If $\beta \in \Phi_{\im}$, then $\mathfrak{g}^{(\beta)}$ is infinite-dimensional
(see \cite[\S 11.9]{Kac90}).

Let $\alpha\in \Phi_{\re}$ and let $\beta\in \Phi$, with $\alpha$ and $\beta$ non-proportional. Then there exist
nonnegative integers $p$ and $q$ such that
\[
p-q=\beta(\alpha^\vee),
\]
and
\[
\beta+k\alpha\in \Phi
\qquad\Longleftrightarrow\qquad
-p\le k\le q,\quad k\in\mathbb Z
\]
(cf.~ \cite[Proposition 5.1]{Kac90}). In this case, the sequence
\[
\beta-p\alpha,\ \ldots,\ \beta-\alpha,\ \beta,\ \beta+\alpha,\ \ldots,\ \beta+q\alpha
\]
is called the \emph{$\alpha$-string through $\beta$}.

\subsection{Closed subsets of root systems and regular subalgebras}

Let $\Phi$ be the root system of $\mathfrak g$ with respect to a Cartan
subalgebra $\mathfrak h$. A subset $T\subseteq\Phi$ is \emph{closed} if
$\alpha,\beta\in T$ and $\alpha+\beta\in\Phi$ imply $\alpha+\beta\in T$.
For any $S\subseteq\Phi$, the \emph{closure of $S$}, denoted $[S]$, is defined by 
\[
[S]\coloneqq \text{the smallest closed subset of $\Phi$ containing $S$}.
\]

A subalgebra $\mathfrak{s}\subseteq\mathfrak{g}$ is 
\emph{regular} if it is normalized by some Cartan subalgebra
of $\mathfrak{g}$. Fix such a Cartan subalgebra $\mathfrak{h}$.
Then $\mathfrak{s}$ admits a root space decomposition
\[
\mathfrak{s}
=
(\mathfrak{s}\cap \mathfrak{h})
\oplus
\bigoplus_{\alpha\in \Phi} (\mathfrak{s}\cap \mathfrak{g}_\alpha).
\]
Accordingly, we associate to $\mathfrak{s}$ a subset $T\subseteq \Phi$
defined by
\[
\alpha\in T \quad \Longleftrightarrow \quad
\mathfrak{s}\cap \mathfrak{g}_\alpha \neq \{0\}.
\]
The associated subset $T$ need not be closed, as the following example illustrates.

\begin{example}\label{ex:support-not-closed}
Let $\mathfrak g$ be the affine Kac--Moody algebra of type $A_1^{(1)}$,
with generalized Cartan matrix
\[
A=\begin{pmatrix}2&-2\\-2&2\end{pmatrix},
\]
simple roots $\Pi=\{\alpha_0,\alpha_1\}$, and null root
\[
\delta=\alpha_0+\alpha_1.
\]
Then
\[
\Phi_{\re}=\{\pm\alpha_1+n\delta\mid n\in\mathbb Z\},
\qquad
\Phi_{\im}=\{k\delta\mid k\in\mathbb Z\setminus\{0\}\}.
\]
Choose a nonzero vector $x\in\mathfrak g_\delta$, and consider
the one-dimensional subalgebra
\[
\mathfrak s=\mathbb Cx.
\]
Since $x$ is a root vector, $\mathfrak s$ is normalized by
$\mathfrak h$, so it is regular. Its associated subset is
\[
T=\{\alpha\in\Phi\mid
\mathfrak s\cap\mathfrak g_\alpha\neq\{0\}\}
=\{\delta\}.
\]
However,
\[
\delta\in T,\qquad
\delta+\delta=2\delta\in\Phi,\qquad
2\delta\notin T.
\]
Therefore $T$ is not closed.
\unskip\hfill$\lozenge$
\end{example}

For real roots $\alpha\in\Phi_{\re}$, the root space
$\mathfrak{g}_\alpha$ is one-dimensional, and hence
\[
\mathfrak{s}\cap \mathfrak{g}_\alpha
=
\{0\}
\quad \text{or} \quad
\mathfrak{g}_\alpha.
\]
Thus, on real roots, the subset $T$ records exactly which root spaces
are contained in $\mathfrak{s}$.

\begin{remark}[Regular subalgebras of semisimple Lie algebras]
In the finite-dimensional semisimple Lie algebra case, one may define a regular subalgebra $\mathfrak{s}$ (normalized by the Cartan subalgebra $\mathfrak{h}$) as one that can be written as 
\begin{equation}\label{semireg}
\mathfrak{s}=\mathfrak{t} \oplus \bigoplus_{\alpha\in T} \mathfrak{g}_\alpha,
\end{equation}
for some closed subset $T\subseteq \Phi$, and subalgebra $\mathfrak{t}\subseteq \mathfrak{h}$, such that
$[\mathfrak{g}_\alpha, \mathfrak{g}_{-\alpha}]\subseteq \mathfrak{t}$ whenever 
both $\alpha$ and $-\alpha$ are in $T$. In the semisimple case,  this is equivalent to requiring $\mathfrak{s}$ to be normalized by a Cartan subalgebra. 

However, in the symmetrizable Kac--Moody setting, the two notions don't necessarily coincide. A subalgebra of the form in Eq. \eqref{semireg} is regular, but a regular subalgebra is not necessarily in this form. In particular, for an imaginary root $\beta$, a regular subalgebra $\mathfrak{s}$ may satisfy
$\{0\} \subsetneq (\mathfrak{s}\cap \mathfrak{g}_\beta)  \subsetneq \mathfrak{g}_\beta$, which cannot happen in the finite-dimensional semisimple case, for which imaginary roots don't exist.
\end{remark}

A subset $S \subseteq \Phi$ is called \emph{symmetric} if
\[
\alpha \in S \quad \Longrightarrow \quad -\alpha \in S.
\]
Equivalently, $S = -S$.

In addition to closed subsets of $\Phi$, we will also require the notion
of real closed subsets  and closure within the real roots $\Phi_{\re}$. A subset
$T\subseteq\Phi_{\re}$ is \emph{real closed} if $\alpha,\beta\in T$ and
$\alpha+\beta\in\Phi_{\re}$ imply $\alpha+\beta\in T$
(cf.~\cite{HabibKusVenkatesh2025}). For $S\subseteq\Phi$, the
\emph{real closure} is defined by
\[
\operatorname{cl}_{\re}(S)
\coloneqq
\text{the smallest real closed subset of $\Phi_{\re}$ containing
$S\cap\Phi_{\re}$}.
\]
In particular, the definition implies  $\operatorname{cl}_{\re}(S)
=
\operatorname{cl}_{\re}(S\cap \Phi_{\re})$, and  $\operatorname{cl}_{\re}(S)\subseteq \Phi_{\re}$.

\section{Wide subalgebras of symmetrizable Kac--Moody algebras}\label{lwidesec}

In this section, we establish the main results of the paper. We first prove a sufficient condition for wideness (Theorem~\ref{lwide}) for regular subalgebras containing a Cartan subalgebra, based on an extension of Schur's lemma (Theorem~\ref{nzero}). We then establish, in the affine case, a necessary condition for wideness (Theorem~\ref{affine-necessary}), which notably does not require the assumption that $\mathfrak{h}\subseteq \mathfrak{s}$. The proofs rely on three key ingredients: the Idempotent Criterion Lemma (Lemma~\ref{newidempotentlemma}), the extension of Schur's lemma, and the result on the closure of finite sums of roots in the affine case (Proposition~\ref{affine-sum-prop}). The first result is well known in the semisimple setting and extends naturally to symmetrizable Kac--Moody algebras; we include a proof for completeness.

\begin{lemma}[Idempotent criterion lemma]\label{newidempotentlemma}
Let $\mathfrak{s}$ be a subalgebra of a symmetrizable Kac--Moody algebra 
$\mathfrak{g}$, and let $V$ be an irreducible integrable highest weight module.
Then $V$ is $\mathfrak{s}$-indecomposable if and only if 
$(\End \, V)^{\mathfrak{s}}$ contains no nontrivial idempotents.
\end{lemma}
\begin{proof}
($\Longrightarrow$)
Assume $V$ is $\mathfrak{s}$-indecomposable.
Let $F \in (\End\, V)^{\mathfrak{s}}$ be idempotent, so that $F^2 = F$. Since $x \cdot F = x \circ F - F \circ x = 0$ for all $x \in \mathfrak{s}$,
we have $F \circ x = x \circ F$, that is, $F$ commutes with every element 
of $\mathfrak{s}$. 

We then have a vector space decomposition
\[
V = \operatorname{im}(F) \oplus \ker(F).
\]
Indeed, let  $v\in V$, then $v=F(v)+(v-F(v))$, where $F(v)\in \operatorname{im}(F)$, and $v-F(v)\in \ker(F)$, since $F$ is idempotent. This implies $V=\operatorname{im}(F)+\ker(F)$. Further, if $v\in \operatorname{im}(F) \cap \ker(F)$, then $F(v)=0$ and $F(u)=v$ for some $u\in V$. Since $F$ is idempotent, we have $v=F(u)=F^2(u)=F(v)=0$, so that $\operatorname{im}(F) \cap \ker(F)=\{0\}$. Thus $V = \operatorname{im}(F) \oplus \ker(F)$ as vector spaces.

Further, the sum is direct as $\mathfrak{s}$-submodules: if $F(v)=0$ and $x\in\mathfrak{s}$, 
then $F(xv)=xF(v)=0$, so $\ker(F)$ is an $\mathfrak{s}$-submodule; and if $v=F(u)$, then 
$xv = xF(u)=F(xu)\in\operatorname{im}(F)$, so $\operatorname{im}(F)$ is an $\mathfrak{s}$-submodule. 

Indecomposability then implies that one of these summands must be $\{0\}$.
If $\operatorname{im}(F) = \{0\}$, then $F = 0$.
If $\ker(F) = \{0\}$, then the decomposition
$V = \ker(F) \oplus \operatorname{im}(F)$
gives $\operatorname{im}(F) = V$ directly, so $F$ is surjective. Combined with 
$\ker(F)=\{0\}$, $F$ is bijective, and multiplying $F^2 = F$ by $F^{-1}$ 
yields $F = \mathrm{Id}_V$. Thus $F$ is a trivial idempotent.

\vspace{4pt}
\noindent
($\Longleftarrow$)
Suppose $V$ is $\mathfrak{s}$-decomposable.
Then there exists a decomposition
\[
V = U \oplus W
\]
with $U$ and $W$ nonzero $\mathfrak{s}$-submodules.
Let $F:V\to V$ be the projection onto $U$ along $W$, that is
$F(u+w)=u$ for $u\in U$, $w\in W$.
Then $F^2=F$ and $F\notin\{0,\mathrm{Id}_V\}$.
Moreover, for $x\in\mathfrak{s}$ and $v=u+w\in V$, with $u\in U$ and $w\in W$,
\[
F(xv)
=
F(xu+xw)
=
xu
=
xF(v),
\]
since $U$ and $W$ are $\mathfrak{s}$-submodules, so $x\cdot F = x\circ F - F\circ x = 0$, so that $F\in(\End\, V)^{\mathfrak{s}}$. Thus $F$ is a nontrivial idempotent in  $(\End\, V)^{\mathfrak{s}}$.
\end{proof}

For our next result, we extend Schur's lemma from a symmetrizable Kac--Moody algebra to certain regular subalgebras.

\begin{theorem}[Schur-type rigidity for regular subalgebras]\label{nzero}
Let $\mathfrak{g}$ be a symmetrizable Kac--Moody algebra with Cartan
subalgebra $\mathfrak{h}$ and root system $\Phi$, and let $T\subseteq\Phi$
be the subset associated with a regular subalgebra $\mathfrak{s}$, such that $
\mathfrak{h}\subseteq \mathfrak{s}$. Let $V$ be an irreducible integrable highest weight $\mathfrak{g}$-module. If
\[
\Pi \subseteq \operatorname{cl}_{\re}(T\cup(-T)),
\]
then
\[
(\End V)^{\mathfrak{s}}
=
\mathbb{C}\Id_V .
\]
\end{theorem}
\begin{proof}
    By Schur's Lemma \cite[Lemma 9.3]{Kac90}, $(\End V )^{\mathfrak{g}} =\mathbb{C} \Id_{V}$. Hence, it suffices to show  $(\End V)^{\mathfrak{s}} = (\End V)^{\mathfrak{g}}$. 
Since we necessarily have $(\End V )^{\mathfrak{g}}\subseteq (\End V)^{\mathfrak{s}}$, we need to establish $(\End V)^{\mathfrak{s}} \subseteq (\End V)^{\mathfrak{g}}$. 

Towards this end, let $F\in (\End V)^{\mathfrak{s}}$. We aim to show that
$F\in(\End V)^{\mathfrak{g}}$, that is, $\mathfrak{g}\cdot F=0$. Since
$\mathfrak{h}\subseteq \mathfrak{s}$, it follows that $\mathfrak{h}\cdot F=0$.
Since $T \cup (-T)$ is symmetric, its real closure $\operatorname{cl}_{\re}(T \cup (-T))$ is also symmetric. Therefore, if $\alpha \in \Pi \subseteq \operatorname{cl}_{\re}(T \cup (-T))$, then $-\alpha \in \operatorname{cl}_{\re}(T \cup (-T))$, and hence $\pm \Pi \subseteq \operatorname{cl}_{\re}(T \cup (-T))$. 

Since $\mathfrak{g}$
is generated by $\mathfrak{h}$ together with the simple root spaces
$\mathfrak{g}_{\pm\alpha}$ ($\alpha\in\Pi$), it suffices to show that
\[
\mathfrak{g}_\alpha \cdot F=0
\qquad
\text{for all } \alpha \in \operatorname{cl}_{\re}(T\cup(-T)),
\]
which we establish in two cases.

\vspace{0.7em}

\noindent \underline{Case 1. $\alpha \in (T \cup (-T))\cap \Phi_{\re}$}\,:
If $\alpha\in T\cap \Phi_{\re}$, then 
$\mathfrak{g}_\alpha \cdot F=0$, since $\mathfrak{g}_{\alpha}\subseteq \mathfrak{s}$ (by definition of $T$) and $F\in (\End V)^{\mathfrak{s}}$.
Now suppose $\alpha \in (-T) \cap \Phi_{\re}$, so that $-\alpha\in T\cap \Phi_{\re}$, and $\mathfrak{g}_{-\alpha}\cdot F=0$. Let $e_\alpha\in\mathfrak{g}_\alpha$ and $e_{-\alpha}\in\mathfrak{g}_{-\alpha}$ be chosen such that
$\alpha^\vee=[e_\alpha,e_{-\alpha}]$,  to form the $\mathfrak{sl}_2$-triple 
\[\mathfrak{s}_\alpha \coloneqq\langle e_{\alpha}, \,e_{-\alpha},\, \alpha^\vee \rangle \cong \mathfrak{sl}_2.\] 

As $V$ is integrable, both
$e_\alpha$ and $e_{-\alpha}$ act locally nilpotently on $V$. In particular, 
$V$ decomposes as a direct sum of finite-dimensional irreducible
$\mathfrak{sl}_2$-modules \cite[Proposition 3.6]{Kac90}.

Let $W(n)$ be an arbitrary $\mathfrak{s}_\alpha$-module of highest weight $n\in \mathbb{Z}_{\ge 0}$ in this decomposition. To show 
$e_\alpha\cdot F=0$ (and hence $\mathfrak{g}_\alpha \cdot F=0$), it suffices to show $e_\alpha F(w) = F(e_\alpha w)$ for an arbitrary basis element $w\in W(n)$.
The $\mathfrak{s}_\alpha$-module $W(n)$ is $(n+1)$-dimensional, and has a basis 
\[
w_0,\,w_1,\cdots,\,w_n,
\]
satisfying \cite[Lemma 7.2]{humphreys}
\begin{equation}\label{humph}
\begin{aligned}
    \alpha^\vee ~w_i &= (n-2i)~ w_i \quad (0\le i\le n);\\
    e_{-\alpha} ~w_i &= (i+1) ~w_{i+1} \quad (0\le i\le n-1), \quad\ e_{-\alpha}~ w_n=0; \;\,\text{and}\\
    e_{\alpha} ~w_i &= (n-i+1)~ w_{i-1}\quad  (1\le i\le n), \quad e_\alpha~ w_0=0.
\end{aligned}
\end{equation}
We will show $e_{\alpha}F(w_m)= F(e_\alpha w_m)$ for all $0\le m \le n$. We proceed by induction on $m$.

First observe that if $F(w_0)=0$, then $F(w_m)=0$ for all $0\le m \le n$, since 
\[
m!\, F(w_m)=F(e_{-\alpha}^m w_0)= e_{-\alpha}^m F(w_0)=0.\] 

Moreover, if $F(w_0)\ne 0$, then $F(w_m)\ne 0$ for all $0\le m \le n$. This follows since, if nonzero, \(F(w_0)\) would be a weight vector for the
\(\mathfrak{s}_\alpha\)-action of weight \(n\). Hence, by \(\mathfrak{sl}_2\)-theory and the fact
that \(V\) decomposes into a direct sum of finite-dimensional irreducible
\(\mathfrak{s}_\alpha\)-modules, we obtain
\[
0 \neq e_{-\alpha}^m F(w_0)
= F(e_{-\alpha}^m w_0)
= m!\,F(w_m)
\qquad (0 \le m \le n).
\]

Hence,  we may assume $F(w_m)\ne 0$ ($0\le m \le n$), for otherwise $F(w_m)=0$ ($0\le m \le n$), and the result holds trivially. 

\vspace{0.7em} 

\noindent \underline{Base cases. $m=0, 1$}~: We have $e_\alpha w_0=0$, so that $F(e_\alpha w_0)=0$. Set $y=F(w_0)$. Since $F$ commutes with $\alpha^\vee$
and $e_{-\alpha}$, we have
\[
\alpha^\vee y=ny,
\qquad
e_{-\alpha}^{\,n+1}y
=F(e_{-\alpha}^{\,n+1}w_0)=0.
\]
In the decomposition of $V$ into finite-dimensional irreducible
$\mathfrak{s}_\alpha$-modules, these conditions force every
nonzero component of $y$ to be a highest weight vector of
highest weight $n$. Hence $e_\alpha y=0$, as required.

This completes the argument when $n=0$. Assume now that $n\ge1$.

Now we will show that $e_\alpha F(w_1)=F(e_\alpha w_1)$. Since
$F(e_\alpha w_1)=n F(w_0)$ by Eq. 
\eqref{humph}, we must show  $e_\alpha F(w_1)$ also equals $nF(w_0)$:
\begin{align*}
    e_\alpha F(w_1) &= e_\alpha F(e_{-\alpha} w_0)\\
    &= e_\alpha e_{-\alpha} F(w_0)\\
    &= \alpha^\vee F(w_0) +\cancelto{0}{e_{-\alpha}\, e_{\alpha}\, F(w_0)}\\
    &= nF(w_0),
\end{align*}
as required. Thus, we have established
\[
e_\alpha F(w_0)=F(e_\alpha w_0) \quad \text{and}\quad e_\alpha F(w_1)=F(e_\alpha w_1).
\]

\vspace{0.7em} 

\noindent \underline{Induction Step. $m \rightarrow m+1$}~: Assume that $e_{\alpha} F(w_m)=F(e_{\alpha} w_m)$, for $1\le m < n$. Then
\begin{align*}
    &~ e_{-\alpha} e_{\alpha} F(w_m) = e_{-\alpha} F(e_{\alpha} w_m)\\
    \Longrightarrow &~ -\alpha^\vee F(w_m) +e_\alpha e_{-\alpha} F(w_m)= F(e_{-\alpha} e_\alpha w_m) \quad \text{(since $[e_\alpha, e_{-\alpha}]=\alpha^\vee$)}\\
     \Longrightarrow &~ -F(\alpha^\vee w_m) +e_\alpha F(e_{-\alpha} w_m) =
     F(-\alpha^\vee w_m +e_{\alpha} e_{-\alpha} w_m) \\
 \Longrightarrow &~ -(n-2m)F(w_m) +(m+1) e_\alpha F( w_{m+1}) \\
 &~ = -(n-2m)F(w_m) + (m+1)F(e_\alpha w_{m+1}) \quad \text{(by Eq.\eqref{humph})} \\
 \Longrightarrow &~  e_\alpha F( w_{m+1}) 
 = F(e_\alpha w_{m+1}),\quad \text{as required}.
\end{align*}

Hence, in Case 1, we have established that 
\[
\mathfrak{g}_\alpha \cdot F=0,\quad \text{for all}\quad \alpha \in (T \cup (-T))\cap \Phi_{\re}.
\]

\vspace{0.7em}

\noindent \underline{Case 2. $\alpha \in \operatorname{cl}_{\re}(T\cup(-T)) \subseteq \Phi_{\re}$}\,: First, we  construct  $\operatorname{cl}_{\re}(T\cup(-T))$ recursively as follows. Define 
\begin{align*}
\operatorname{cl}_{\re}(T\cup(-T))_0 &= (T \cup (-T))\cap \Phi_{\re},\\
\operatorname{cl}_{\re}(T\cup(-T))_1 &= \{\alpha+\beta \in \Phi_{\re} \setminus \operatorname{cl}_{\re}(T\cup(-T))_0 ~|~  \alpha,\, \beta \in \operatorname{cl}_{\re}(T\cup(-T))_0  \},\\
\operatorname{cl}_{\re}(T\cup(-T))_n&= \{\alpha+\beta \in \Phi_{\re} \setminus \cup_{i=0}^{n-1}\operatorname{cl}_{\re}(T\cup(-T))_i ~|~  \alpha,\, \beta \in \cup_{i=0}^{n-1} \operatorname{cl}_{\re}(T\cup(-T))_i  \},
 \end{align*}
for $n>1$. Then,
\[
\operatorname{cl}_{\re}(T\cup(-T)) = \cup_{i=0}^\infty \operatorname{cl}_{\re}(T\cup(-T))_i.
\]

Having constructed $\operatorname{cl}_{\re}(T\cup(-T))$, we now show $\mathfrak{g}_{\alpha} \cdot F =0$ for all $\alpha \in \operatorname{cl}_{\re}(T\cup(-T))$ by induction on the level $n$. If $\alpha \in \operatorname{cl}_{\re}(T\cup(-T))_0  = (T \cup (-T))\cap \Phi_{\re}$, then
we have already shown $\mathfrak{g}_{\alpha} \cdot F =0$ in Case 1, hence the base case is established. 

Assume $\mathfrak{g}_{\beta} \cdot F =0$ for all $\beta \in \operatorname{cl}_{\re}(T\cup(-T))_i$, $i \leq n-1$, and consider $\alpha \in \operatorname{cl}_{\re}(T\cup(-T))_n$. 
We must have $\alpha=\gamma+\nu$ for some $\gamma,\, \nu \in \cup_{i=0}^{n-1} \operatorname{cl}_{\re}(T\cup(-T))_i$. Since $\gamma,\nu$, $\gamma+\nu \in \Phi_{\re}$, we have
\[
[\mathfrak{g}_{\gamma}, \mathfrak{g}_{\nu}] = \mathfrak{g}_{\gamma+\nu}.
\]
Hence,
\[
\mathfrak{g}_{\gamma+\nu} \cdot F
=
[\mathfrak{g}_{\gamma}, \mathfrak{g}_{\nu}] \cdot F
=
\mathfrak{g}_{\gamma} \cdot (\mathfrak{g}_{\nu}\cdot F)
-
\mathfrak{g}_{\nu} \cdot (\mathfrak{g}_{\gamma}\cdot F)
=0,
\]
since $\mathfrak{g}_{\gamma} \cdot F =0$ and $\mathfrak{g}_{\nu} \cdot F =0$ by the induction hypothesis.

\vspace{2mm}

Hence, the above two cases establish that $\mathfrak{g}_\alpha \cdot F =0$, for all $\alpha\in \operatorname{cl}_{\re}(T\cup(-T))$, as required.
\end{proof}

\begin{remark}[The zero weight space of $\End(V)$]
Theorem \ref{nzero} assumes that $\mathfrak{h} \subseteq \mathfrak{s}$. 
If we drop this hypothesis, while keeping the others, then we still obtain
\[
(\End V)^{\mathfrak{s}}_0
=
\mathbb{C}\Id_V .
\]
Indeed, this follows by applying Theorem \ref{nzero} to $\mathfrak{s}+\mathfrak{h}$, together with the observation that
$
(\End V)^{\mathfrak{s}}_0 = (\End V)^{\mathfrak{s} + \mathfrak{h}}$.
\end{remark}

We now state and prove one of our  main theorems, which adapts the sufficient condition of Corollary~3.10 
of~\cite{DouglasRepka2025} from the semisimple setting to the symmetrizable 
Kac--Moody setting.

\begin{theorem}[Sufficient condition for wideness]\label{lwide}
Let $\mathfrak{g}$ be a symmetrizable Kac--Moody algebra with Cartan
subalgebra $\mathfrak{h}$, root system $\Phi$, and simple roots $\Pi$. Let $T\subseteq \Phi$
be the subset associated with a regular subalgebra $\mathfrak{s}$, such that $\mathfrak{h}\subseteq \mathfrak{s}$.
Then,
\vspace{2pt}
\[
\Pi \subseteq \operatorname{cl}_{\re}(T\cup(-T))
\quad\Longrightarrow\quad
\mathfrak{s}\ \text{is wide}.
\]
\end{theorem}

\begin{proof} Suppose $\Pi\subseteq \operatorname{cl}_{\re}(T\cup(-T))$, $\mathfrak{h}\subseteq \mathfrak{s}$, and let
\[F \in (\mathrm{End}\,V)^{\mathfrak{s}}\] be idempotent, where $V$ is an arbitrary irreducible integrable highest weight module of $\mathfrak{g}$.  
Since $F \in (\mathrm{End}\,V)^{\mathfrak{s}}$, then $F \in \mathbb{C} \Id_{V}$ by Theorem \ref{nzero}. 
Then, since $F$ is idempotent, $F=0$ or $F=\Id_{V}$. That is, $F$ is a trivial idempotent. Hence, by Lemma \ref{newidempotentlemma}, $\mathfrak{s}$ is wide.
\end{proof}

\vspace{0.5em}

\begin{proposition}[Closure under finite root sums in affine type]\label{affine-sum-prop}
Let $\Phi$ be an affine root system, and let $S$ be a closed subset of $\Phi$.
If $\beta_1,\dots,\beta_n\in S$ and
\[
\beta_1+\cdots+\beta_n\in \Phi,
\]
then
\[
\beta_1+\cdots+\beta_n\in S.
\]
\end{proposition}

\begin{proof}
We proceed by induction on $n$. The case $n=1$ is trivial. For $n=2$, the result follows directly from $S$ being closed. 

Now let
\[
\eta=\beta_1+\cdots+\beta_n\in \Phi,
\]
and assume the statement holds for sums of fewer than $n$ roots.

If there exist $i\neq j$ such that $\beta_i+\beta_j\in \Phi$, then since $S$ is closed,
\[
\beta_i+\beta_j\in S,
\]
and replacing $\beta_i,\beta_j$ by $\beta_i+\beta_j$ reduces the length of the sum by one.
Hence, by induction, $\eta\in S$.

If there exist $i\neq j$ such that $\beta_i+\beta_j=0$, then deleting this pair reduces the length
of the sum by two without changing the total sum, so again induction gives $\eta\in S$.

Thus we may assume that
\[
\beta_i+\beta_j\notin \Phi\cup\{0\}\qquad (i\neq j).
\]

It therefore suffices to show that there exists $m\in \{1,\dots,n\}$ such that
\[
\eta-\beta_m\in \Phi,
\]
for then
\[
\eta-\beta_m=\sum_{i\neq m}\beta_i
\]
is a sum of fewer than $n$ elements of $S$, and the induction hypothesis yields
$\eta-\beta_m\in S$. Since also $\beta_m\in S$ and
\[
(\eta-\beta_m)+\beta_m=\eta\in \Phi,
\]
the closedness of $S$ gives $\eta\in S$.

We consider two cases.

\medskip

\noindent\underline{Case 1. $\eta\in \Phi_{\re}$}\,:
Since
\[
(\eta\mid \eta)=\sum_{i=1}^n (\eta\mid \beta_i)>0,
\]
there exists $m\in \{1,\dots,n\}$ such that
\[
(\eta\mid \beta_m)>0.
\]
As $\eta$ is real, Proposition 5.1 of \cite{Kac90} implies that
\[
\beta_m-\eta\in \Phi\cup\{0\}.
\]
Hence
\[
\eta-\beta_m\in \Phi\cup\{0\}.
\]
If $\eta-\beta_m=0$, then $\eta=\beta_m\in S$, and we are done. Thus we may assume that
\[
\eta-\beta_m\in \Phi.
\]
\medskip

\noindent\underline{Case 2. $\eta\in \Phi_{\im}$}\,:
Since $\Phi$ is affine, we have
\[
\eta=t\delta
\qquad\text{for some } t\in \mathbb Z\setminus\{0\}.
\]

\smallskip

\noindent\underline{Case 2.1. $\beta_k\in \Phi_{\re}$ for some $k$}\,:
Since $\beta_k+\beta_i\notin \Phi\cup\{0\}$ for $i\ne k$, we have
\[
(\beta_k\mid \beta_i)\ge 0 \qquad (i\ne k)
\]
by \cite[Proposition 5.1]{Kac90}. Therefore
\[
(\beta_k\mid \eta)
=\sum_{i=1}^n (\beta_k\mid \beta_i)
=(\beta_k\mid \beta_k)+\sum_{i\ne k}(\beta_k\mid \beta_i)>0,
\]
since $(\beta_k\mid \beta_k)>0$.
On the other hand, $\eta=t\delta$ for some $t\neq 0$, and in affine type
\[
(\delta\mid \beta_k)=0.
\]
Hence
\[
(\beta_k\mid \eta)=t(\beta_k\mid \delta)=0,
\]
a contradiction.
Thus, this case cannot occur.

\smallskip

\noindent\underline{Case 2.2. $\beta_k\in \Phi_{\im}$ for all $k$} \,:
Then each $\beta_k=m_k\delta$ with $m_k\neq 0$. But for any $i\neq j$,
\[
\beta_i+\beta_j=(m_i+m_j)\delta \in \Phi\cup\{0\},
\]
contrary to our standing assumption that
\[
\beta_i+\beta_j\notin \Phi\cup\{0\}.
\]
Thus this case cannot occur.

\medskip

In every possible case, we find $m$ such that $\eta-\beta_m\in \Phi$, and the induction argument
completes the proof.
\end{proof}

\begin{remark}
Proposition~\ref{affine-sum-prop} strengthens the notion of closed subsets
in affine root systems by showing that  closed subsets are
closed under arbitrary finite root sums that remain roots. This is a key ingredient
in the proof of Theorem~\ref{affine-necessary} below. 

The analogous closure property holds in finite type root systems
(cf.~\cite[Lemma 3.2]{DouglasRepka2025}). In general symmetrizable Kac--Moody type,
the argument used to establish this result, in the proof of Proposition \ref{affine-sum-prop},   does not hold.
\end{remark}

\begin{definition}[$\mathfrak{s}$-weight graphs]
Let $V$ be a $\mathfrak g$-module with weight decomposition
$V=\bigoplus_{\lambda} V_\lambda$, and let $\mathfrak s$ be a regular subalgebra with associated subset $T\subseteq \Phi$. The \emph{$\mathfrak s$-weight graph} of $V$ has vertex set
\[
\{\lambda \in \mathfrak h^* \mid V_\lambda \neq 0\},
\]
with an edge between $\lambda$ and $\mu$ if $\mu-\lambda \in T\cup(-T)$.
\end{definition}

\vspace{0.5em}

\begin{theorem}[Necessary condition for wideness]\label{affine-necessary}
Let $\mathfrak g$ be an affine Kac--Moody algebra with Cartan subalgebra
$\mathfrak h$, and let $\mathfrak s\subseteq \mathfrak g$ be a regular
subalgebra normalized by $\mathfrak h$. Let $T\subseteq \Phi$ be the associated subset. Then
\[
\mathfrak{s}\ \text{is wide}
\quad\Longrightarrow\quad
[T\cup(-T)] = \Phi.
\]
\end{theorem}

\begin{proof}
Set
\[
S=[T\cup(-T)].
\]
Then $S$ is symmetric and closed.

By way of contradiction, assume that $S\neq \Phi$. Then there exists $\alpha_k\in \Pi$
such that
\[
\alpha_k\notin S.
\]
Indeed, if $\Pi\subseteq S$, then, since $S$ is symmetric, also $-\Pi\subseteq S$.
By Proposition~\ref{affine-sum-prop}, every positive real root, being a sum of
simple roots, belongs to $S$, and hence $\Phi_{\re}\subseteq S$. Moreover, in
affine type the null root $\delta$ is a positive root and a sum of simple roots,
so $\delta\in S$. Applying Proposition~\ref{affine-sum-prop} again, we obtain
$m\delta\in S$ for every $m\in\mathbb Z\setminus\{0\}$, hence $\Phi_{\im}\subseteq S$.
Therefore $\Phi= S$, a contradiction.

Now consider the irreducible integrable highest weight module
\[
V=V(\omega_k),
\]
where $\omega_k$ is the $k$th fundamental weight. Since $\omega_k(\alpha_k^\vee)=1$, the weight $\omega_k-\alpha_k$ occurs in $V$ (indeed, $f_k v_{\omega_k}\neq 0$).

We claim that $\omega_k$ and $\omega_k-\alpha_k$ lie in distinct connected
components of the $\mathfrak s$-weight graph of $V$.

Suppose not. Then there exists a path
\[
\omega_k=\lambda_0,\lambda_1,\dots,\lambda_r=\omega_k-\alpha_k
\]
in the $\mathfrak s$-weight graph. By definition of the graph, for each
$j=1,\dots,r$ there exists $\beta_j\in T\cup(-T)\subseteq S$ such that
\[
\lambda_j-\lambda_{j-1}=\beta_j.
\]
Summing gives
\[
-\alpha_k
=(\omega_k-\alpha_k)-\omega_k
=\beta_1+\cdots+\beta_r.
\]
Since $-\alpha_k\in \Phi$ and each $\beta_j\in S$, Proposition~\ref{affine-sum-prop}
implies that
\[
-\alpha_k\in S.
\]
As $S$ is symmetric, this yields $\alpha_k\in S$, a contradiction.

Thus $\omega_k$ and $\omega_k-\alpha_k$ lie in different connected components of the $\mathfrak{s}$-weight graph of $V$, so the graph has at least two connected components. For each connected component $C$ of the $\mathfrak s$-weight graph, set
\[
V_C \coloneqq \bigoplus_{\lambda \in C} V_\lambda.
\]
Since $\mathfrak s$ is regular and normalized by $\mathfrak{h}$, we have
\[
\mathfrak s=(\mathfrak s\cap \mathfrak h)\oplus \bigoplus_{\beta\in T}(\mathfrak s\cap \mathfrak g_\beta).
\]
The subalgebra $\mathfrak s\cap \mathfrak h$ preserves each weight space, and for $x\in \mathfrak s\cap \mathfrak g_\beta$ with $\beta\in T$, one has
\[
x(V_\lambda)\subseteq V_{\lambda+\beta}.
\]
If $V_{\lambda+\beta}\neq 0$, then $\lambda$ and $\lambda+\beta$ are adjacent in the $\mathfrak s$-weight graph, so $\lambda+\beta$ lies in the same connected component as $\lambda$. Hence each $V_C$ is an $\mathfrak s$-submodule. Since the connected components partition the vertex set, we obtain a direct sum decomposition
\[
V=\bigoplus_C V_C.
\]
As the graph has at least two connected components, this yields a nontrivial decomposition of $V$ as an $\mathfrak s$-module, contradicting wideness.

Therefore,
\[
[T\cup(-T)]=\Phi.
\]
\end{proof}

We conclude this section with two straightforward examples illustrating the main theorems. The first applies Theorem~\ref{lwide} to construct a wide regular subalgebra of an affine Kac--Moody algebra. The second applies Theorem~\ref{affine-necessary} to exhibit a regular subalgebra of the same algebra that is not wide.

\begin{example}[A wide subalgebra in type $A_1^{(1)}$]
Let $\mathfrak g$ be the affine Kac--Moody algebra of type $A_1^{(1)}$,
with generalized Cartan matrix
\[
A=
\begin{pmatrix}
2 & -2\\
-2 & 2
\end{pmatrix},
\]
simple roots $\Pi=\{\alpha_0,\alpha_1\}$, and null root
\[
\delta=\alpha_0+\alpha_1.
\]
Then
\[
\Phi_{\re}=\{\pm\alpha_1+n\delta \mid n\in\mathbb Z\},
\qquad
\Phi_{\im}=\{k\delta \mid k\in\mathbb Z\setminus\{0\}\}.
\]
Define
\[
\mathfrak s
=
\mathfrak h
\oplus
\mathfrak g_{\alpha_1}
\oplus
\mathfrak g_{-\alpha_0}.
\]
Since
\[
\alpha_1+\alpha_1\notin \Phi,\qquad
(-\alpha_0)+(-\alpha_0)\notin \Phi,\qquad
\alpha_1-\alpha_0=2\alpha_1-\delta\notin \Phi,
\]
the two root spaces commute. Thus $\mathfrak s$ is a subalgebra,
and it is regular since it is normalized by $\mathfrak h$.
Its associated subset is
\[
T=\{\alpha_1,\,-\alpha_0\}.
\]
We have
\[
T\cup(-T)=\{\pm\alpha_1,\pm\alpha_0\},
\]
and therefore
\[
\Pi\subseteq T\cup(-T)\subseteq \operatorname{cl}_{\re}(T\cup(-T)).
\]
It follows from Theorem~\ref{lwide} that $\mathfrak s$ is wide.
\unskip\hfill$\lozenge$
\end{example}

\begin{example}[A non-wide subalgebra in type $A_1^{(1)}$]
Let $\mathfrak g$ be the affine Kac--Moody algebra of type $A_1^{(1)}$, as above. Define
\[
\mathfrak s
=
\mathfrak h
\oplus
\mathfrak g_{\alpha_1}.
\]
Since $\alpha_1+\alpha_1\notin \Phi$, the root space
$\mathfrak g_{\alpha_1}$ is abelian. Thus $\mathfrak s$ is a subalgebra,
and it is regular since it is normalized by $\mathfrak h$.
Its associated subset is
\[
T=\{\alpha_1\}.
\]
We then have
\[
T\cup(-T)=\{\pm\alpha_1\},
\qquad
[T\cup(-T)]=\{\pm\alpha_1\}\neq \Phi.
\]
It follows from Theorem~\ref{affine-necessary} that $\mathfrak s$ is not wide.
\unskip\hfill$\lozenge$
\end{example}

\section{Conclusion}\label{conclusion}

In this article, we studied wide regular subalgebras of symmetrizable
Kac--Moody algebras, extending aspects of the theory from the
finite-dimensional semisimple setting in \cite{DouglasRepka2025}. In the Cartan-regular case
$\mathfrak{h}\subseteq \mathfrak{s}$, we proved that
\[
\Pi \subseteq \operatorname{cl}_{\re}(T\cup(-T))
\quad\Longrightarrow\quad
\mathfrak{s}\ \text{is wide}.
\]

In the affine case, we established the necessary condition
\[
\mathfrak{s}\ \text{is wide}
\quad\Longrightarrow\quad
[T\cup(-T)]=\Phi,
\]
which does not require the assumption that $\mathfrak{h}\subseteq \mathfrak{s}$.

A main ingredient is an extension of Schur's lemma, which replaces complete reducibility in this setting. In particular, under the same real-root closure condition, the invariant endomorphisms satisfy $(\End V)^{\mathfrak{s}}=\mathbb{C}\Id_V$ for any irreducible integrable highest weight module $V$, whenever $\mathfrak{h}\subseteq \mathfrak{s}$.

The approach differs fundamentally from that in the finite-dimensional semisimple setting of \cite{DouglasRepka2025}. In the semisimple case, the argument relies on the complete reducibility of $\End(V)$ and the finiteness of its weight decomposition, neither of which holds in the Kac--Moody setting. The presence of imaginary roots introduces additional complications and requires new structural tools.

Several natural questions remain open. First, it would be desirable
to determine necessary and sufficient conditions for an arbitrary
regular subalgebra to be wide, including those that do not necessarily
contain a Cartan subalgebra. In the present work, we obtain a sufficient
condition under the assumption $\mathfrak{h}\subseteq\mathfrak{s}$,
and, in the affine case, a necessary condition that does not require
this assumption. Bridging this gap in the general symmetrizable
Kac--Moody setting remains an open problem and may require additional
conditions on the Cartan component $\mathfrak{s}\cap\mathfrak{h}$
or on imaginary roots. In particular, further extensions of
Schur's lemma may be needed.

More specifically, we may ask whether, under the real-root closure
condition $\Pi\subseteq\operatorname{cl}_{\mathrm{re}}(T\cup(-T))$,
it is sufficient for wideness that $\mathfrak{s}$ contain
a complement of $\mathfrak{h}\cap[\mathfrak{g},\mathfrak{g}]$
in $\mathfrak{h}$. In the finite-dimensional semisimple case,
this complement is zero, and no additional condition on the Cartan
component is required~\cite{DouglasRepka2025}.

A second direction for future work is the study of \emph{narrow}
subalgebras (i.e., those for which every irreducible
integrable highest weight module of dimension greater than one becomes decomposable upon restriction)
and of \emph{regular extreme} algebras, that is, algebras for which every
regular subalgebra is either narrow or wide. Developing such a theory
for symmetrizable Kac--Moody algebras would parallel the semisimple
results established in~\cite{DouglasRepka2025}.

\section*{Acknowledgments}

A. Douglas was supported in part by PSC-CUNY Award
\#GR-00017525. A. Douglas used ChatGPT for proofreading and editing.
The authors take full responsibility for all mathematical results, arguments, and exposition in the manuscript.

\end{document}